\documentclass[12pt, reqno]{amsart}
\usepackage{lineno,hyperref}
\usepackage{cases}
\usepackage{latexsym,epsfig}
\usepackage[square, comma, sort&compress, numbers]{natbib}%²Î¿¼ÎÄÏ×Á¬Ðø±àºÅ
\usepackage{float}
\usepackage{pifont}
\usepackage{bbding}
\usepackage{mathrsfs}
\usepackage{subfigure}
\usepackage{caption}
\usepackage{geometry}
\usepackage{amsmath, amsthm, amscd, amsfonts, amssymb, graphicx, color}
\usepackage{pgfplots}%%create plots
 \usepackage{ulem}
\pgfplotsset{compat=newest}
\newtheorem{thm}{Theorem}
\newtheorem{theorem}{Theorem}%[section]
{Lemma}
\newtheorem{lem}%[theorem]
{Lemma}
{Proposition}
\newtheorem{corollary}%[theorem]
{Corollary}
\theoremstyle{definition}
\newtheorem{definition}{Definition}

\newtheorem{examples}{Examples}

\newtheorem{remark}{Remark}
\newtheorem{conj}{Conjecture}

\theoremstyle{remark}
\numberwithin{equation}{section}

\newcommand{\ID}{\mathbb{D}}
\newcommand{\IC}{{\mathbb C}}
\newcommand{\ds}{\displaystyle}
\newcommand{\dist}{{\operatorname{dist}}}

\def\be{\begin{equation}}
\def\ee{\end{equation}}

\newcounter{minutes}
\divide\time by 60
\newcounter{hours}
\multiply\time by 60 \addtocounter{minutes}{-\time}
\begin{document}
\thispagestyle{empty} \setcounter{page}{1}

%\noindent\parbox{2.85cm}{\includegraphics*[keepaspectratio=true,scale=1.75]{BJMA.jpg}}
%\noindent\parbox{4.85in}{\hspace{0.1mm}\\[1.5cm]\noindent
%Banach J. Math. Anal. 0 (0000), no. 0, 00--00\\
%$\frac{\rule{4.55in}{0.05in}}{{}}$\\
%{\footnotesize
%\textcolor[rgb]{0.65,0.00,0.95}{\textsc{\textbf{\large{B}}anach
%\textbf{\large{J}}ournal of \textbf{\large{M}}athematical
%\textbf{\large{A}}analysis}}\\
%ISSN: 1735-8787 (electronic)\\
%\textcolor[rgb]{0.00,0.00,0.84}{\textbf{http://www.math-analysis.org }}\\
%$\frac{{}}{\rule{4.55in}{0.05in}}$}\\[.5in]}
%\linenumbers

\title[Bohr radius for subordination and $K$-quasiconformal harmonic mappings]
{Bohr radius for subordination and $K$-quasiconformal harmonic mappings}

%=========================================================================
\thanks{%$^\dagger$
File:~\jobname .tex,
          printed: \number\day-\number\month-\number\year,
          \thehours.\ifnum\theminutes<10{0}\fi\theminutes}
%=========================================================================

%%=========================================================================
%\thanks{%$^\dagger$
%File:~\jobname .tex,
%          printed: \number\day-\number\month-\number\year,
%          \thehours.\ifnum\theminutes<10{0}\fi\theminutes}
%%=========================================================================
%\author[I. R. Kayumov]{Ilgiz R Kayumov }
%\address{I. R. Kayumov
%\vskip.03in Kazan Federal University, Kremlevskaya 18, 420008 Kazan, Russia.}
%\email{\textcolor[rgb]{0.00,0.00,0.84}{ikayumov@kpfu.ru}}

\author[Z. Liu]{ZhiHong Liu }
\address{Z. Liu, College of Science, Guilin University of Technology,
Guilin 541004, Guangxi, People's Republic of China.}
%\author[Z. Liu]{ZhiHong Liu }
%\address{Z. Liu, College of Mathematics, Honghe University,
%Mengzi 661199, Yunnan, People's Republic of China.}
%\vskip.03in School of Mathematics and Econometrics, Hunan University,
%Changsha 410082, Hunan, People's Republic of China.}
\email{\textcolor[rgb]{0.00,0.00,0.84}{liuzhihongmath@163.com}}

\author[S. Ponnusamy]{Saminathan Ponnusamy}
\address{S. Ponnusamy, Department of Mathematics,
Indian Institute of Technology Madras, Chennai-600 036, India}
\email{\textcolor[rgb]{0.00,0.00,0.84}{samy@iitm.ac.in}}

%\dedicatory{This paper is dedicated to Professor ABCD}



\begin{abstract}
%We consider the class of locally univalent and sense-preserving  $K$-quasiconformal harmonic mappings $f=h+\overline{g}$
%defined on the unit disk $\mathbb{D}$, and determine the Bohr radius when the analytic part $h$ subordinate
%to a convex or a starlike or a univalent function.
The present article concerns the Bohr radius for $K$-quasiconformal  sense-preserving harmonic mappings $f=h+\overline{g}$
in the unit disk $\mathbb{D}$ for which the analytic
part $h$ is subordinated to some analytic function $\varphi$, and the purpose is to look into two
cases: when $\varphi$ is convex, or a general univalent function in $\ID$. The results
state that if
$h(z) =\sum_{n=0}^{\infty}a_n z^n$ and $g(z)=\sum_{n=1}^{\infty}b_n z^n$, then
$$\sum_{n=1}^{\infty}(|a_n|+|b_n|)r^n\leq \dist (\varphi(0),\partial\varphi(\ID)) ~\mbox{ for $r\leq r^*$}
$$
and give estimates for the largest possible $r^*$ depending only on the geometric
property of $\varphi (\ID)$ and the parameter $K$. Improved versions of the theorems are
given for the case when $b_1 = 0$ and corollaries are drawn for the case when
$K\rightarrow \infty$.

\end{abstract}

\maketitle
%---------------------------------------------------------------------------------------------------------------%
\section{Introduction and Preliminaries}
%---------------------------------------------------------------------------------------------------------------%
A classical theorem of Bohr states that \cite{Bohr1914}, if $f$ is a bounded analytic function on the open unit disk $\mathbb{D}:=\{z\in\mathbb{C}:|z|<1\}$, with
power series of the form $f(z)=\sum_{n=0}^{\infty}a_n z^n$, then
\be\label{sub}
B_{f}(r):=\sum_{n=0}^{\infty}|a_n| r^n\leq \|f\|_\infty ~\mbox{ for all }~ |z|=r\leq \frac{1}{3}
\ee
and the constant $1/3$, often called the Bohr radius, cannot be improved. This inequality
known as Bohr's inequality, was originally obtained in 1914 by H. Bohr  for $0\leq r\leq 1/6$.
The fact that the inequality is
actually true for $0\leq r\leq 1/3$ and that $1/3$  is the best possible constant was obtained independently
by M.~Riesz, I.~Schur and F.~Wiener. Bohr's and Wiener's proofs can be found in \cite{Bohr1914}.
Several improved versions of the Bohr inequality are established.

For example, in a related development, Kayumov and Ponnusamy \cite{KayPon3} gave several improved versions of Bohr's inequality. Some of them may now be recalled.

\begin{thm}\label{KayPon8-Additional4}
Suppose that $f(z) = \sum_{k=0}^\infty a_k z^k$ is analytic in $\ID$, $|f(z)| \leq 1$ in $\ID$ and  $S_r$ denotes the area of the
image of the subdisk $|z|<r$ under the mapping $f$. Then
\begin{equation*}\label{Eq_Th3}
%B_1(r):=
\sum_{k=0}^\infty |a_k|r^k+\frac{16}{9}\left (\frac{S_r}{\pi}\right )  \leq 1 ~\mbox{ for  }~ r \leq \frac{1}{3},
\end{equation*}
and the constants $1/3$ and $16/9$ cannot be improved. Moreover,
\begin{equation*}\label{Eq2_Th3}
%B_2(r):=
|a_0|^2+\sum_{k=1}^\infty |a_k|r^k+\frac{9}{8}\left (\frac{S_r}{\pi}\right )  \leq 1 ~\mbox{ for  }~ r \leq \frac{1}{2},
\end{equation*}
and the constants $1/2$ and $9/8$ cannot be improved.
\end{thm}

\begin{thm}\label{KayPon8-1to3}
Suppose that $f(z) = \sum_{k=0}^\infty a_k z^k$ is analytic in $\ID$ and $|f(z)| \leq 1$ in $\ID$. Then we have
\begin{enumerate}
\item $\ds
%\be\label{KayPon8-eq6}
|a_0|+\sum_{k=1}^\infty \left(|a_k|+\frac{1}{2}|a_k|^2\right)r^k  \leq 1 ~\mbox{ for  }~ r \leq \frac{1}{3},
$
and the constants $1/3$ and $1/2$  cannot be improved.
\item $\ds \sum_{k=0}^\infty |a_k|r^k + |f(z)-a_0|^2 \leq 1 ~\mbox{ for  }~ r \leq \frac{1}{3}, $
and the constant $1/3$  cannot be improved.
\item $\ds
|f(z)|^2+\sum_{k=1}^\infty |a_k|^2r^{2k} \leq 1 ~\mbox{ for  }~ r \leq \sqrt{\frac{11}{27}},
$
and the constant $11/27$ cannot be improved.
\end{enumerate}
\end{thm}

For the last two decades, Bohr's inequality has been revived and improved in many ways due to the
discovery of generalizations to domains in $\mathbb{C}^n$ and to more abstract settings. For background
information about this result and further work related to Bohr's phenomenon, we refer to
the recent surveys by Abu-Muhanna et al.~\cite{Ali2016}, B\'en\'eteau et al. \cite{BDK5},
Ismagilov et al. \cite{IsmaKayKayPon1},  Kayumova et al. \cite{KayKayPon2} and the references therein.
Some of the recent results from \cite{Ali2017,Kayumov1,Kayumov1708.05578,KayPon3} are included in the latest two surveys.
More generally, harmonic version of Bohr's inequality was discussed by Kayumov et al. in \cite{Kayumov2}
which will also be recalled below. For certain other results on harmonic Bohr's inequality, we refer to
\cite{Kayumov2,Evdoridis1709.08944}. We refer to \cite{LiPo-RM18} for Bohr's inequality for
the class of harmonic $\nu$-Bloch-type mappings as a generalization of harmonic $\nu$-Bloch mappings and to \cite{AlkKayPon1} for
the class of quasi-subordinations.

%In this article, we shall investigate Bohr's radius for locally univalent and sense-preserving harmonic mappings defined on the unit disk $\mathbb{D}$.
A harmonic mapping $f$ defined on $\ID$ is a complex-valued function $f = u + iv$,
where $u$ and $v$ are real-valued harmonic functions of $\ID$. It follows that $f$ admits the representation $f=h+\overline{g}$,
where $h$ and $g$ are analytic in $\ID$ known as the analytic and co-analytic parts of $f$, respectively.
We follow the convention that $g(0) = 0$ so that the representation $f=h+\overline{g}$ is unique and is called the canonical representation
of $f$ and thus $h$ and $g$ admit power series expansions of the form
$$h(z) =\sum_{n=0}^{\infty}a_n z^n ~\mbox{ and }~ g(z)=\sum_{n=1}^{\infty}b_n z^n, \quad z\in\ID.
$$
A locally univalent harmonic function $f$ in $\mathbb{D}$ is said to be sense-preserving if the Jacobian $J_f (z)$ of $f$ given by
$J_f (z)=|h'(z)|^2-|g'(z)|^2$, is positive in $\mathbb{D}$; or equivalently, its dilatation
$\omega =g'/h'$ is an analytic function in $\ID$ which maps  $\mathbb{D}$ into itself (See \cite{Lewy1936} or \cite{Duren1983}).

If a locally univalent and sense-preserving harmonic mapping $f=h+\overline{g}$ satisfies the condition
$$\left|\frac{g'(z)}{h'(z)}\right|\leq k<1 ,
$$
then $f$ is called   $K$-quasiconformal harmonic mapping on $\mathbb{D}$, where $K=\frac{1+k}{1-k}\geq 1$
(cf. \cite{Martio1968,Kalaj2008}, and also \cite{Zhu2016} for some recent investigation on harmonic $K$-quasiconformal self-mapping of $\ID$).
Obviously $k\rightarrow 1$ corresponds to the case $K\rightarrow \infty$.
Harmonic extension of the classical Bohr theorem was established in \cite{Kayumov2}. For example, they proved the
following results.

\begin{thm} \label{KSP4-th3}
Suppose that $f(z) = h(z)+\overline{g(z)}=\sum_{n=0}^\infty a_n z^n+\overline{\sum_{n=1}^\infty b_n z^n}$ is a sense-preserving
$K$--quasiconformal  harmonic mapping of the disk $\ID$, where $h$ is a bounded function in $\ID$.
Then we have
\begin{enumerate}
\item
$\ds \sum_{n=0}^\infty |a_n|r^n+\sum_{n=1}^\infty |b_n|r^n \leq ||h||_{\infty}~\mbox{ for  }~r \leq  \frac{K+1}{5K+1}.
$
The constant $(K+1)/(5K+1)$ is sharp.
\item
$\ds
|a_0|^2+\sum_{n=1}^\infty (|a_n|+|b_n|)r^n \leq ||h||_{\infty}~\mbox{ for  }~ r \leq \frac{K+1}{3K+1}.
$
The constant $(K+1)/(3K+1)$ is sharp.
\end{enumerate}
\end{thm}

\begin{thm}\label{KSP4-th2}
Suppose that either $f=h+g$ or $f = h+\overline{g}$, where $h(z)=\sum_{n=1}^\infty a_n z^n$ and $g(z)=\sum_{n=1}^\infty b_n z^n$
are bounded analytic functions in $\ID$. Then
$$\sum_{n=1}^\infty (|a_n|+|b_n|)r^n \leq \max\{||h||_{\infty}, ||g||_{\infty}\} ~\mbox{ for  }~r \leq \sqrt{\frac{7}{32}}.
$$
This number $\sqrt{7/32}$ is sharp.
\end{thm}

The purpose of this article is to determine the Bohr radius for the class of
%locally univalent and
$K$-quasiconformal sense-preserving  harmonic mappings $f=h+\overline{g}$, where $h$ is subordinate to $\varphi$, where
$\varphi$ is either a general function in the convex family or in the univalent family.

The paper is organized as follows. In Section \ref{lem2}, we present main definitions and necessary lemmas that are required to state and prove
our main results. Section \ref{sec3-main} begins with examples containing test functions for which our main results could be used to derive several
new theorems and corollaries, and then we state and prove our main theorems and several of their consequences.
More precisely Theorems \ref{thQHC1} and \ref{thHC} generalize Theorem \ref{KayPon8-Additional4}(1) whereas  Theorems \ref{thULH} and \ref{thHS} essentially deal with the case when the subordinating function is univalent instead of convex. The article concludes with a conjecture.

\section{Necessary Lemmas}\label{lem2}

We need to recall some basic notions and results on subordination.

\begin{definition}
Let $\varphi $ and $g$ be analytic in $\ID$ with $\varphi (0)=g(0)$. Then we say that $g$ is subordinate to $\varphi$
(written by $g\prec \varphi$ or $g(z)\prec \varphi (z)$) if
$$g(z)=\varphi (\omega(z)) ~\mbox{ for $|z|<1$}
$$
for some analytic function $\omega$ on $\ID$ with $|\omega(z)|\leq |z|$ for $z\in\ID$. When $\varphi$ is univalent, $g\prec \varphi$ precisely
when $\varphi(0) = g(0)$ and $g(\ID)\subset \varphi (\ID)$.
\end{definition}

For basic details and results on subordination classes, see for example~\cite[Chapter 6]{Duren1983} or~\cite[p. 35]{Pommerenke1975}. Let $\mathcal{S}$ denote the class of all
univalent analytic mappings $\varphi$ on $\ID$ normalized by $\varphi(0)=0$ and $\varphi '(0)=1$.
Denote by $\mathcal{S}^{*}$ and $\mathcal{C}$ the subclass of $\mathcal{S}$ of mappings that map $\ID$ onto starlike
and convex domains, respectively. See \cite{Duren1983} for details on these classes and many other related subclasses of $\mathcal{S}$.
If $\varphi$ is univalent, then the following coefficient inequalities are well-known.

\begin{thm}\label{SCSB} (L. de Branges' Theorem)
%{\rm (\cite[Theorem 6.4]{Duren1983})}
Suppose that $g\prec \varphi $ and $g(z)=\sum_{n=1}^{\infty}b_n z^n$. If $\varphi\in \mathcal{S}$, then $|b_n|\leq n$ for $n\geq 2 $.
% We have
\end{thm}
%{\color{red}{
Because $\varphi\in \mathcal{C}$ if and only if $z\varphi '\in \mathcal{S^{*}}$, and $\mathcal{S^{*}}\subset \mathcal{S}$,
Theorem \ref{SCSB}, in particular gives the following:
\begin{enumerate}
\item [(1)] if $\varphi\in \mathcal{C}$, then $|b_n|\leq 1$ for $n\geq 2 $;
\item[(2)] if $\varphi\in \mathcal{S^{*}}$, then $|b_n|\leq n$ for $n\geq 2 $.
\end{enumerate}

%\begin{lemma}\label{RogoCj}{\rm (Rogosinski Conjecture~\cite[p. 196]{Duren1983})}
%If $g(z)=\sum_{n=1}^{\infty}b_n z^n$ is analytic in $\ID$ and $g\prec f$ for some $f\in \mathcal{S}$, then $|b_n|\leq n$ for $n=1,2,\cdots$.
%\end{lemma}

Throughout this paper, we denote the class of all analytic functions $g$ in $\ID$ subordinate to a fixed univalent function $\varphi$ in $\ID$ by
$$S(\varphi)=\left\{g: \, g\prec \varphi\right\}.
$$
%{\color{red}where $\varphi(z)=\sum_{n=0}^{\infty} c_n z^n$.}
We say that the family $S(\varphi)$ has Bohr's phenomenon if for any $ g\in S(\varphi)$ and $g(z)=\sum_{n=0}^{\infty} b_n z^n\prec \varphi (z)$
there is an $r_\varphi$, $0<r_\varphi \leq 1$, such that (see \cite{Abu, Ali2016})
\begin{equation}\label{SBHR}
\sum_{n=1}^{\infty} \left|b_n z^n\right|\leq \dist (\varphi (0), \partial \Omega  ) ~\mbox{for $|z|<r_\varphi$,}
\end{equation}
where $\dist (\varphi (0), \partial \Omega )$ denotes the distance between $\varphi (0)$ and
the boundary $\partial \Omega $ of $\Omega =\varphi (\ID)$.

We observe that if  $\varphi(z)=(\alpha -z)/(1-\overline{\alpha}z)$ with $|\alpha|<1$, then $\Omega=\varphi (\ID) =\ID $,
$\varphi (0)=\alpha$ and $\dist (\varphi (0),\partial \Omega  )=1-|\alpha|=1-|b_0|$ so that \eqref{SBHR}
(and hence \eqref{sub}) holds with $r_\varphi =1/3$.

We can easily to obtain the following two lemmas  from \cite[p.~195-196]{Duren1983} (see also \cite{Ali2016,Pommerenke1975}).

\begin{lem}\label{lemUd}
Let $\varphi$ be an analytic univalent map from $\ID$ onto a simply connect domain $\Omega=\varphi (\ID)$. Then
\begin{equation*}
\frac{1}{4}|\varphi '(0)|\leq \dist (\varphi (0),\partial \Omega)\leq |\varphi '(0)|.
\end{equation*}
If $g(z)=\sum_{n=0}^{\infty}b_n z^n\prec \varphi (z)$, then
\begin{equation*}
  |b_n|\leq n|\varphi'(0)|\leq 4 n \,\dist (\varphi(0),\partial \Omega).
\end{equation*}
\end{lem}

\begin{lem}\label{lemCd}
Let $\varphi$ be an analytic univalent map from $\ID$ onto a convex domain $\Omega=\varphi(\ID)$. Then
\begin{equation}\label{Bohr-eq1x}
\frac{1}{2}|\varphi'(0)|\leq \dist (\varphi(0),\partial \Omega)\leq |\varphi'(0)|.
\end{equation}
If $g(z)=\sum_{n=0}^{\infty}b_n z^n\prec \varphi(z)$, then
\begin{equation*}
  |b_n|\leq|\varphi'(0)|\leq 2 \,\dist (\varphi(0),\partial \Omega).
\end{equation*}
\end{lem}

Particularly, the well-known Growth Theorem implies that if $\varphi\in \mathcal{S}$ then
\begin{equation}\label{GTHMS}
\frac{1}{4}\leq \dist (0,\partial \varphi(\ID))\leq 1
\end{equation}
and if $\varphi\in \mathcal{C}$ then
\begin{equation}\label{GTHMC}
\frac{1}{2}\leq \dist (0,\partial \varphi(\ID))\leq 1.
\end{equation}
See~\cite[Theorems 2.6 and 2.15]{Duren1983} or \cite[p. 22]{Pommerenke1975}.
%{\color{red}{
Note that \eqref{GTHMS} and \eqref{GTHMC} follow from Lemmas \ref{lemUd} and \ref{lemCd}, respectively.
%}}

The following lemma plays an important role in the proof of our results.
%\begin{lemma}{\rm (\cite[Lemma 1]{Kayumov2})}\label{AnBn}
%Suppose that $h(z)=\sum_{n=0}^{\infty}a_n z^n$ and $g(z)=\sum_{n=0}^{\infty}b_n z^n$ are two analytic functions in the unit disk $\ID$, such that $|g'(z)|\leq |h'(z)|$ in $\ID$. Then
%\begin{equation}\label{eqlm}
%\sum_{n=1}^{\infty} |b_n|^2 r^n\leq \sum_{n=1}^{\infty} |a_n|^2 r^n~ \mbox{ for }~ |z|=r<1.
%\end{equation}
%\end{lemma}

 \begin{lem}\label{lem:Kayumov2}
{\rm (see \cite[Lemma 2.1]{Kayumov2})}
Suppose that $h(z)=\sum_{n=0}^\infty a_n z^n$ and $g(z)=\sum_{n=0}^\infty b_n z^n$ are two analytic functions
in the unit disk $\ID$ such that $|g'(z)| \leq k |h'(z)|$ in $\ID$ and for some $k\in [0,1]$. Then
\begin{equation*}%\label{eqR1}
\sum_{n=1}^\infty |b_n|^2r^n \leq k^2\sum_{n=1}^\infty |a_n|^2r^n ~\mbox{ for }~ |z|=r<1.
\end{equation*}
\end{lem}

%In this article, we determine the Bohr radius for different classes of $K$-quasiconformal harmonic mappings $f=h+\overline{g}$ especially when
%the analytic part $h$ is subordinate to a  convex or a starlike or a univalent function.

%---------------------------------------------------------------------------------------------------------------%
\section{Main results and their proofs}\label{sec3-main}
%---------------------------------------------------------------------------------------------------------------%

Before we state and prove our main theorems, it is worth pointing out that our approach provides many results by
different choices of $\varphi$ in the main theorems. To demonstrate this, we first present a set of test functions
for which our results are applicable.

\begin{examples}
\begin{enumerate}
\item[\bf (a)] For $\lambda \in (0,1]$ and $\alpha \in \IC\backslash\{0\}$, consider
 $$\varphi(z)=%\frac{\alpha z}{1+(1+\lambda)e^{i\theta}z+\lambda e^{2i\theta}z^2} =
\frac{\alpha z}{(1+z)(1+\lambda z)}, \quad z\in\ID.
$$
Then it is easy to see that $\dist (\varphi(0),\partial\varphi(\ID))=|\alpha |/(2(1+\lambda))$, because
$$\frac{\alpha }{\varphi(z)}=\frac{1}{z} +\lambda z +(1+\lambda)
$$
and for $\lambda \in (0,1)$, $w= \frac{1}{z} +\lambda z$ maps $\ID$ onto the exterior of the ellipse bounded by
$$\frac{U^2}{(1+\lambda)^2} +\frac{V^2}{(1-\lambda)^2}=1.
$$
Also, we see that $\varphi$ is univalent in $\ID$.

\item[\bf (b)] For $\lambda \in [0,1)$,  consider the univalent function
 $$\varphi(z)=%\frac{\alpha z}{1+(1+\lambda)e^{i\theta}z+\lambda e^{2i\theta}z^2} =
\frac{z}{1-2\lambda z +z^2}= \frac{1}{\frac{1}{z}+z-2\lambda } , \quad z\in\ID.
$$
Then it can be easily seen that  $\varphi(z)$ maps $\ID$ onto the complement of segment $(-1/(2(1+\lambda)),1/(2(1-\lambda)))$
and thus,
$$\dist (\varphi(0),\partial\varphi(\ID))=1/(2(1+\lambda)).
$$

\item[\bf (c)]
For $a\in {\mathbb R}\backslash \{0\}$ and $c>0$, consider
$$\varphi(z)=a\left (\frac{1+z}{1-z}\right ) +2(\sqrt{c^2+a^2} \, -a)\frac{z}{1-z^2} , \quad z\in\ID.
$$
Using the range of the function $\psi (z)=2cz/(1-z^2)$, it can be easily shown that $\varphi(z)$ maps $\ID$ onto the
complex plane with slits along half-lines ${\rm Re}\, w =0$ and $|{\rm Im}\, w|\geq c:=c(a)$ such that $\varphi(0)=a$
and
$$\dist (\varphi(0),\partial\varphi(\ID))= \dist (a,\partial\varphi(\ID))=\sqrt{c^2+a^2}.
$$
Obviously, $\varphi$ is univalent and starlike in $\ID$.

In particular, for $a>0$ or $-n/2<a<0$, set $\sqrt{c^2+a^2} \, -a=n$, where $n>0$. Then we obtain that the function
$$\varphi_{a,n}(z)=\frac{a(1+z)}{1-z} +\frac{2nz}{1-z^2}
$$
maps $\ID$ onto the complex plane with slits along half-lines ${\rm Re}\, w =0$ and $|{\rm Im}\, w|\geq c:=c(a,n)=\sqrt{n(n+2a)}$ such that $\varphi(0)=a$
and
$$\dist (\varphi(0),\partial\varphi(\ID))= n+a.
$$
\item[\bf (d)]  For $\lambda \in \IC$, consider the function
$$\varphi(z)=\frac{z-\lambda z^2}{(1-z)^2}.
$$
We see that it is univalent in $\ID$ if and only if $|2\lambda -1|\leq 1$. The function $\varphi(z)$ in general is not starlike, for example,
for $\lambda =1/2$, this function is known to be close-to-convex (univalent) but is not starlike in $\ID$.

In particular, for $\lambda \in [0,1]$, we may write
$\varphi$ as
$$\varphi(z)=\frac{1}{4(1-\lambda)}(\zeta ^2-1),
$$
where $\zeta =\psi (z)=\frac{1+(1-2\lambda) z}{1-z}$ and $w =\Psi (\zeta )=\zeta^2.$
It follows that $\zeta =\psi (z) $ maps the disk $\ID$ onto the half-plane ${\rm Re}\, w>\lambda$ and $ w =\Psi (\zeta)$ maps the half-plane ${\rm Re}\, w>\lambda$ onto the  parabolic region
$$x<\lambda ^2 -\frac{y^2}{4\lambda ^2}.
$$
Consequently, for $\lambda \in [0,1]$, $\varphi$ maps $\ID$ onto a parabolic region such that
$$\dist (\varphi(0),\partial\varphi(\ID))= -\frac{1-\lambda ^2}{4(1-\lambda)}=-\frac{1+\lambda }{4}.
$$
%Moreover, if $|2\lambda -1|= 1$ ($\lambda \neq 1$), we may write $\lambda =e^{i\alpha}\cos \alpha$ ($\alpha \in (0,\pi)$) and in this case,
%$\varphi(\ID)$ is the whole complex plane minus a straight line cut where the slit end point is
%$$e_0=-\frac{1+i\cot \alpha}{4}.
%$$

\item[\bf (e)] For $|\lambda| <1$, consider
$$\varphi(z)=\lambda+\frac{2}{\pi}\log \left (\frac{1+z\xi}{1-z} \right ), \quad z\in\ID,
$$
where $\xi=e^{-i\pi{\rm Im}\,\lambda}$. Then it is a simple exercise to show that
$\varphi(z)$ maps $\ID$ onto the strip $\Omega =\{w:\,|{\rm Im}\, w|<1\}$ with $\varphi(0)=\lambda$
and $\varphi '(0)=\frac{2}{\pi}(1+\xi)$ such that
$$\dist (\varphi(0),\partial\varphi(\ID))=1-|{\rm Im}\,\lambda|.
$$
Note that $\varphi$ is convex.

\item[\bf (f)] For ${\rm Re}\,\lambda >0$, consider
$$\varphi(z)=\frac{\lambda + \overline{\lambda}z}{1-z}, \quad z\in\ID.
$$
Then we see that $\varphi(z)$ maps $\ID$ onto the right half-plane ${\rm Re}\, w >0$ such that $\varphi(0)=\lambda$ and
$$\dist (\varphi(0),\partial\varphi(\ID))={\rm Re}\,\lambda.
$$
Clearly, $\varphi$ is convex.

\item[\bf (g)] For $\alpha \in [1,2]$, the function
$$\varphi(z)=\frac{1}{2\alpha}\left [\left (\frac{1+z}{1-z} \right )^{\alpha}-1\right ]
$$
is univalent in $\ID$ and
$$\dist (\varphi(0),\partial\varphi(\ID))=\frac{1}{2\alpha}.
$$
\item[\bf (h)] For $\alpha \in [0,1)$, the function
$$\varphi(z)=\frac{z}{(1-z)^{2(1-\alpha)}}
$$
is univalent (and is in fact starlike of order $\alpha$) in $\ID$ and
$$\dist (\varphi(0),\partial\varphi(\ID))=\frac{1}{2^{2(1-\alpha)}}.
$$
\end{enumerate}
\end{examples}
The following result is a generalization of \cite[Theorems 1.1 and 1.3]{Kayumov2} (see also Theorem \ref{KayPon8-Additional4}) for appropriate choices of $\varphi$.

\begin{theorem}\label{thQHC1}
Suppose that $f(z)=h(z)+\overline{g(z)}=\sum_{n=0}^{\infty}a_n z^n+\overline{\sum_{n=1}^{\infty}b_n z^n}$ is a  $K$-quasiconformal
 sense-preserving  harmonic mapping in $\ID$ and $h\prec \varphi$, where $\varphi$ is univalent and convex in $\ID$. Then
\begin{equation*}
\sum_{n=1}^{\infty}(|a_n|+|b_n|)r^n\leq \dist (\varphi(0),\partial\varphi(\ID)) ~\mbox{ for $\ds |z|=r\leq \frac{K+1}{5K+1}$.}
\end{equation*}
The result is sharp.
\end{theorem}
\begin{proof}
By assumption $h\prec \varphi$ and $\varphi (\ID)$ is a convex domain. Then, by Lemma \ref{lemCd}, we have
\begin{equation*}%\label{Bohr-eq1}
  |a_n|\leq 2 \,\dist (\varphi(0),\partial \varphi(\ID)).
\end{equation*}
Consequently,
\begin{equation*}
\sum_{n=1}^{\infty}|a_n|r^n\leq 2 \,\dist (\varphi(0),\partial \varphi(\ID))\sum_{n=1}^{\infty}r^n
\leq 2 \,\dist (\varphi(0),\partial \varphi(\ID))\frac{r}{1-r}.
\end{equation*}
Because $f=h+\overline{g}$ is a $K$-quasiconformal sense-preserving harmonic mapping so that $|g'(z)|\leq k |h'(z)|$ in $\ID$,
where $0\leq k<1$, by Lemma \ref{lem:Kayumov2} and the Cauchy-Schwarz inequality, it follows that
\begin{eqnarray*}
\sum_{n=1}^{\infty}|b_{n}|r^n
&\leq&\sqrt{\sum_{n=1}^{\infty}|b_{n}|^2r^n}\sqrt{\sum_{n=1}^{\infty}r^n}\\
&\leq &k\sqrt{\sum_{n=1}^{\infty}|a_{n}|^2r^n}\sqrt{\sum_{n=1}^{\infty}r^n}\\
&\leq &2 k \,\dist (\varphi(0),\partial \varphi(\ID))\frac{r}{1-r}.
\end{eqnarray*}
Thus, we have
\begin{equation*}
\sum_{n=1}^{\infty}(|a_n|+|b_n|)r^n\leq 2(1+k)\,\dist (\varphi(0),\partial\varphi(\ID))\frac{r}{1-r}
%\leq \dist (\varphi(0),\partial\varphi(\ID))
\end{equation*}
which is less than or equal to $\dist (\varphi(0),\partial \varphi(\ID))$
for $r\leq\frac{1}{3+2k}$. Substituting $k=\frac{K-1}{K+1}$ gives the desired result.

In order to prove the sharpness, we consider
$$\varphi(z)=h(z)=\frac{1}{1-z}=\sum_{n=0}^{\infty}z^n,$$
and $g'(z)=k\lambda h'(z)$, where $\lambda\in\mathbb{D}$. Then it is easy to see that
$$\dist (\varphi(0),\partial \varphi(\ID))=\frac{1}{2}$$
and
$$g(z)=k\lambda \frac{z}{1-z}=k\lambda\sum_{n=1}^{\infty} z^n.$$
So it is a simple exercise to yield
\begin{equation*}
\sum_{n=1}^{\infty}(|a_n|+|b_n|)r^n=\sum_{n=1}^{\infty}(1+k|\lambda|)r^n =(1+k|\lambda|)\frac{r}{1-r}
\end{equation*}
which is bigger than or equal to $1/2$ if and only if
$$r\geq \frac{1}{3+2k|\lambda|}= \frac{K+1}{3K+3+2|\lambda|(K-1)}.
$$
This shows that the number $\frac{K+1}{5K+1}$ cannot be improved since $|\lambda|$ could be chosen so close to $1$ from left.
This completes the proof.
\end{proof}

Also, it is interesting to note
that when $k = 0$ (or, equivalently, $K = 1$) one retrieves Aizenberg's  \cite{Aiz07} result,
according to which for convex functions $\varphi$, the Bohr inequality \eqref{sub} holds with  $1/3$ as its Bohr radius.
Because of its independent interest, it might be worth stating the following two corollaries.

\begin{corollary}\label{thHC1a}
Suppose that $f(z)=h(z)+\overline{g(z)}=\sum_{n=0}^{\infty}a_n z^n+\overline{\sum_{n=1}^{\infty}b_n z^n}$ is a
%locally univalent and
sense-preserving harmonic mapping in $\ID$ and $h\prec \varphi$, where $\varphi(z)$ is univalent and convex in $\mathbb{D}$. Then
\begin{equation*}
\sum_{n=1}^{\infty}(|a_n|+|b_n|)r^n\leq \dist (\varphi(0),\partial\varphi(\ID))
\end{equation*}
for $|z|=r\leq 1/5$. The number $1/5$ is sharp.
\end{corollary}
\begin{proof}
Allow $k=1$ in the proof of Theorem \ref{thQHC1}. Indeed, since $f(z)$ is locally univalent and sense-preserving in $\ID$,
we have $|g'(z)| < |h'(z)|$ in $\ID$ and thus, we can allow $K\rightarrow \infty$ to obtain the desired conclusion.
\end{proof}

If we choose $\varphi(z)=(\alpha -z)/(1-\overline{\alpha}z)$ with $|\alpha|<1$, then $\varphi(0)=\alpha$
and $\dist (\varphi (0),\partial \Omega  )=1-|\alpha|$  and this clearly give the following corollary
(see also \cite{Kayumov2} or Theorem \ref{KSP4-th3} with $K\rightarrow \infty$).

\begin{corollary}
Suppose that $f(z) = h(z)+\overline{g(z)}=\sum_{n=0}^\infty a_n z^n+\overline{\sum_{n=1}^\infty b_n z^n}$ is a
sense-preserving harmonic mapping of the disk $\ID$, where $|h(z)|< 1$ in $\ID$. Then the following sharp inequality hold:
$$\ds |a_0|+\sum_{n=1}^\infty (|a_n|+ |b_n|) r^n \leq 1~\mbox{ for  }~ r \leq \frac{1}{5}.
$$
\end{corollary}

%If $f=h+\overline{g}\in\mathcal{S}_{H}^{0}$, we obtain the following result.
%
%Although Corollary \ref{thQHC2} contains sharp version of case $k=0$ and $k=1$, the radius is not sharp. However,
%in our next theorem, we improve Corollary \ref{thQHC2} by modifying its proof.

\begin{theorem}\label{thHC}
Suppose that $f(z)=h(z)+\overline{g(z)}=\sum_{n=0}^{\infty}a_n z^n+\overline{\sum_{n=2}^{\infty}b_n z^n}$ is a
%locally univalent and
$K$-quasiconformal sense-preserving harmonic mapping in $\ID$ and $h\prec \varphi$, where $\varphi$ is univalent and convex in $\ID$. Then
\begin{equation*}
\sum_{n=1}^{\infty}|a_n|r^n+\sum_{n=2}^{\infty}|b_n|r^n\leq \dist (\varphi(0),\partial\varphi(\ID))
\end{equation*}
for $|z|=r\leq r_c(k)$, where $r_c(k)$ is the positive root of the equation
\begin{equation}\label{eq:r}
\frac{r}{1-r}+\frac{kr^2}{1-r^2}\sqrt{\left(\frac{1+r^2}{1-r^2}\right)\left(\frac{\pi^2}{6}-1\right)}
=\frac{1}{2}
\end{equation}
and $k=(K-1)/(K+1)$. The number $r_c(k)$ cannot be replaced by the number greater than $R:=R(k)$, where $R$ is the positive root of the equation
\begin{equation}\label{eq:R}
\frac{2(1+k)R}{1-R}+2k\log(1-R)=1.
\end{equation}
%for $|z|=r\leq r_c=\frac{1}{4}$.
%{\color{red}{How about sharpness.}}
%The boundary is sharp and attained by a suitable rotation of the right-half plane mapping $l(z)=z/(1-z)$.
\end{theorem}
\begin{proof} As $\varphi(z)$ is analytic and convex in $\ID$, by \eqref{GTHMC} and Lemma~\ref{lemCd}, we have
$$\dist (\varphi(0),\partial\varphi(\ID))\geq \frac{|\varphi'(0)|}{2}\quad{\rm and}\quad |a_n|\leq |\varphi'(0)| ~\mbox{ for $n\geq 1$}
$$
so that
\begin{equation}\label{An}
\sum_{n=1}^{\infty}|a_n|r^n\leq |\varphi'(0)| \sum_{n=1}^{\infty}r^n=|\varphi'(0)| \frac{ r}{1-r}.
\end{equation}

Because $f=h+\overline{g}$ is locally univalent and  $K$-quasiconformal sense-preserving harmonic mapping with $g'(0)=0$,
Schwarz's lemma gives that $\omega=g'/h'$ is analytic in $\mathbb{D}$ and $|\omega(z)|\leq k|z|$ in $\mathbb{D}$. Thus, we have
$$|g'(z)|^2=|\omega(z) h'(z)|^2\leq k^2|z h'(z)|^2.
$$
Integrate this inequality on the circle $|z|=r$, we obtain
$$\sum_{n=2}^{\infty}n^2 |b_n|^2 r^{2(n-1)}\leq k^2r^2\sum_{n=1}^{\infty}n^2 |a_n|^2 r^{2(n-1)}
\leq k^2 |\varphi'(0)|^2 \frac{r^2(1+r^2)}{(1-r^2)^3}.
$$
By using the Cauchy-Schwarz inequality, it follows that
\begin{equation*}%\label{eq:bn}
\sum_{n=2}^{\infty}|b_n|r^n\leq \sqrt{\sum_{n=2}^{\infty}n^2 |b_n|^2 r^{2n}}\sqrt{\sum_{n=2}^{\infty}\frac{1}{n^2}}
\leq k |\varphi'(0)| r^2\sqrt{\frac{1+r^2}{(1-r^2)^3}}\sqrt{\frac{\pi^2}{6}-1}.
\end{equation*}
Consequently, by combining \eqref{An} with the last inequality, and \eqref{Bohr-eq1x}, we find that
\begin{equation*}
\begin{split}
\sum_{n=1}^{\infty}|a_n|r^n+\sum_{n=2}^{\infty}|b_n|r^n
&\leq  \left ( \frac{r}{1-r}+kr^2\sqrt{\frac{1+r^2}{(1-r^2)^3}}\sqrt{\frac{\pi^2}{6}-1}\right )|\varphi'(0)|\\
& \leq  2\left ( \frac{r}{1-r}+\frac{kr^2}{1-r^2}\sqrt{\left(\frac{1+r^2}{1-r^2}\right)\left(\frac{\pi^2}{6}-1\right)}\right )\dist (\varphi(0),\partial\varphi(\ID))\\
& \leq  \dist (\varphi(0),\partial\varphi(\ID)),
\end{split}
\end{equation*}
%{\color{red}{
where the last inequality holds
%}}
if and only if
\begin{equation*}
\begin{split}
\frac{r}{1-r}+\frac{kr^2}{1-r^2}\sqrt{\left(\frac{1+r^2}{1-r^2}\right)\left(\frac{\pi^2}{6}-1\right)}
\leq  \frac{1}{2}.
\end{split}
\end{equation*}
The above inequality holds for $r\leq r_c(k)$, where $r_c(k)$ is the positive root of the equation \eqref{eq:r}.

Finally, we consider the functions
\begin{equation*}
\varphi(z)=h(z)=\frac{1}{1-z}\quad\mbox{and}\quad g'(z)=kz h'(z).
\end{equation*}
Then we find that
$$|a_n|=1\quad \mbox{ for $n\geq 1$ and}\quad |b_n|=\frac{k(n-1)}{n},\quad n\geq 2,
$$
so that
\begin{equation*}
\begin{split}
\sum_{n=1}^{\infty}|a_n|r^n+\sum_{n=2}^{\infty}|b_n|r^n
&=\sum_{n=1}^{\infty} r^n+k\sum_{n=2}^{\infty}\frac{n-1}{n} r^n\\
&=\frac{(1+k)r}{1-r}+k\log(1-r),
\end{split}
\end{equation*}
which is less than or equal to $1/2$ only in the case when  $r\leq R$, where $R=R(k)$ is the positive root of the equation \eqref{eq:R}.
%As in the proof of Theorem \ref{thHC1}, by Lemma~\ref{AnBn},  $|a_n|\leq |\varphi'(0)|$, and the Cauchy-Schwarz inequality, it follows that
%\begin{equation*}
%\begin{split}
%\sum_{n=2}^{\infty}|b_n|r^n
%%&\leq  \sqrt{\sum_{n=2}^{\infty}|b_n|^2 r^n}\sqrt{\sum_{n=2}^{\infty}r^n}\\
%&\leq  \sqrt{\sum_{n=1}^{\infty}|a_n|^2 r^n}\sqrt{\sum_{n=2}^{\infty}r^n}\\
%&\leq  |\varphi'(0)| \sqrt{\sum_{n=1}^{\infty}r^n}\sqrt{\sum_{n=2}^{\infty}r^n}\\
%& =  |\varphi'(0)| \sqrt{\frac{r}{1-r}}\sqrt{\frac{r^2}{1-r}}\\
%& =  |\varphi'(0)| \frac{r\sqrt{r}}{1-r}.
%\end{split}
%\end{equation*}
%Consequently, by combining \eqref{An} with the last inequality, and \eqref{Bohr-eq1x}, we find that
%\begin{equation*}
%\begin{split}
%\sum_{n=1}^{\infty}|a_n|r^n+\sum_{n=2}^{\infty}|b_n|r^n
%&\leq  \left ( \frac{r}{1-r}+\frac{r\sqrt{r}}{1-r}\right )|\varphi'(0)|\\
%& \leq  2\left ( \frac{r(1+\sqrt{r})}{1-r}\right )\dist (\varphi(0),\partial\varphi(\ID))\\
%& \leq  \dist (\varphi(0),\partial\varphi(\ID))
%\end{split}
%\end{equation*}
%if and only if $2r(1+\sqrt{r})\leq 1-r$, i.e., $4 r^3-9 r^2+6 r-1\leq 0$. This gives that  $r\leq 1/4$.
%Hence the Bohr radius $r_c=1/4$ is the solution of
%$$4 r^3-9 r^2+6 r-1=0$$
%in the interval $(0,1)$.
%Finally, it is evident that $r_c$ is attained by suitable rotations of the right half-pane function $l(z)=z/(1-z)$.
\end{proof}

Setting $k=0$ we see that Theorem \ref{thHC} contains the classical Bohr theorem. The case $k=1$ leads to

\begin{corollary}\label{thHCa}
Suppose that $f(z)=h(z)+\overline{g(z)}=\sum_{n=0}^{\infty}a_n z^n+\overline{\sum_{n=2}^{\infty}b_n z^n}$ is a
%locally univalent and
sense-preserving harmonic mapping in $\ID$ and $h\prec \varphi$, where $\varphi$ is univalent and convex in $\ID$. Then
\begin{equation*}
\sum_{n=1}^{\infty}|a_n|r^n+\sum_{n=2}^{\infty}|b_n|r^n\leq \dist (\varphi(0),\partial\varphi(\ID))
\end{equation*}
for $|z|=r\leq r_c= 0.294265\cdots$, where $r_c$ is the positive root of the equation
\begin{equation*}%\label{eq:r}
\frac{r}{1-r}+\frac{r^2}{1-r^2}\sqrt{\left(\frac{1+r^2}{1-r^2}\right)\left(\frac{\pi^2}{6}-1\right)}
=\frac{1}{2}.
\end{equation*}
The number $0.294265\cdots$ cannot be replaced by the number greater than $R=0.299823\cdots$, where $R$ is the positive root of the equation
\begin{equation*}%\label{eq:R}
\frac{4 R}{1-R}+2\log(1-R)=1.
\end{equation*}
%for $|z|=r\leq r_c=\frac{1}{4}$.
%{\color{red}{How about sharpness.}}
%The boundary is sharp and attained by a suitable rotation of the right-half plane mapping $l(z)=z/(1-z)$.
\end{corollary}

\begin{remark}
%{\color{red}{
Corollary \ref{thHCa} shows that the radius $r_c(k)$ obtained in Theorem  \ref{thHC} is close to the sharp value.
%}}
\end{remark}

\begin{theorem}\label{thULH}
Suppose that $f(z)=h(z)+\overline{g(z)}=\sum_{n=0}^{\infty}a_n z^n+\overline{\sum_{n=1}^{\infty}b_n z^n}$ is a
%locally univalent and
$K$-quasiconformal sense-preserving harmonic mapping in $\ID$ and $h\prec \varphi$, where $\varphi$ is analytic and univalent in $\ID$. Then
\begin{equation*}
\sum_{n=1}^{\infty}(|a_n|+|b_n|)r^n\leq \dist (\varphi(0),\partial\varphi(\ID))
\end{equation*}
for $|z|=r\leq r_u $, where $r_u=r_u(k)$ is the root of the equation
\begin{equation*}
(1-r)^2-4r(1+k\,\sqrt{1+r})=0
\end{equation*}
in the interval $(0,1)$ and $k=(K-1)/(K+1)$.
\end{theorem}
\begin{proof}
By the assumption and Lemma \ref{lemUd}, it follows that $|a_n|\leq 4n \,\dist (\varphi(0),\partial \varphi(\ID))$ and thus,
\begin{equation*}
\sum_{n=1}^{\infty}|a_n|r^n\leq 4 \,\dist (\varphi(0),\partial \varphi(\ID))\sum_{n=1}^{\infty}n r^n
= 4 \,\dist (\varphi(0),\partial \varphi(\ID))\frac{r}{(1-r)^2}.
\end{equation*}
Moreover, because $|g'(z)| \leq k|h'(z)|$ in $\ID$, as in the proof of Theorem \ref{thQHC1},
it follows from Cauchy-Schwarz inequality and Lemma \ref{lem:Kayumov2} with $k=1$ that
\begin{equation*}
\begin{split}
\sum_{n=1}^{\infty}|b_n|r^n
%&\leq\sqrt{\sum_{n=1}^{\infty}|b_n|^2 r^n}\sqrt{\sum_{n=1}^{\infty}r^n}\\
%&\leq \sqrt{\sum_{n=1}^{\infty}|a_n|^2 r^n}\sqrt{\sum_{n=1}^{\infty}r^n}\\
&\leq 4k \,\dist (\varphi(0),\partial \varphi(\ID))\sqrt{\sum_{n=1}^{\infty}n^2 r^n}\sqrt{\sum_{n=1}^{\infty}r^n}\\
&=4k \,\dist (\varphi(0),\partial \varphi(\ID))\sqrt{\frac{r(1+r)}{(1-r)^3}}\sqrt{\frac{r}{1-r}}\\
&=4k \,\dist (\varphi(0),\partial \varphi(\ID))\frac{r\sqrt{1+r}}{(1-r)^2}
\end{split}
\end{equation*}
and thus, we have
\begin{equation*}
\sum_{n=1}^{\infty}(|a_n|+|b_n|)r^n\leq 4 \,\dist (\varphi(0),\partial \varphi(\ID))\frac{r(1+k\,\sqrt{1+r})}{(1-r)^2},
\end{equation*}
which is less than or equal to $\dist (\varphi(0),\partial \varphi(\ID))$ if and only if
\begin{equation*}
\frac{r(1+k\,\sqrt{1+r})}{(1-r)^2}\leq \frac{1}{4}.
\end{equation*}
This gives $|z|=r\leq r_u $, where $r_u=r_u(k)$ is as in the statement.
% the unique root of the equation
%\begin{equation*}
%(1-r)^2-4r(1+\sqrt{1+r})=0
%\end{equation*}
%in the interval $(0,1)$.
\end{proof}

\begin{remark}
Theorem \ref{thULH} for $k=0$ reduces to a result of Abu-Muhanna \cite{Abu} with the sharp Bohr radius as $3-2\sqrt{2}=0.17157...$.
\end{remark}

\begin{corollary}\label{thULHa}
Suppose that $f(z)=h(z)+\overline{g(z)}=\sum_{n=0}^{\infty}a_n z^n+\overline{\sum_{n=1}^{\infty}b_n z^n}$ is a
%locally univalent and
sense-preserving harmonic mapping in $\ID$ and $h\prec \varphi$, where $\varphi$ is analytic and univalent in $\ID$. Then
\begin{equation*}
\sum_{n=1}^{\infty}(|a_n|+|b_n|)r^n\leq \dist (\varphi(0),\partial\varphi(\ID))
\end{equation*}
for $|z|=r\leq r_u= 0.099064\cdots$, where $r_u$ is the root of the equation
\begin{equation*}
(1-r)^2-4r(1+\sqrt{1+r})=0
\end{equation*}
in the interval $(0,1)$.
\end{corollary}
\begin{proof}
Allow $k\rightarrow 1$ in Theorem \ref{thULH}.
\end{proof}

\begin{remark}
When $\varphi$ in Corollary \ref{thULHa} is univalent and $b_1=0$, then the result can be improved (see also Corollary \ref{thHSa}).
\end{remark}

Our next result is to improve Theorem \ref{thULH} when $g'(0)=b_1=0$.

\begin{theorem}\label{thHS}
Suppose that $f(z)=h(z)+\overline{g(z)}=\sum_{n=1}^{\infty}a_n z^n+\overline{\sum_{n=2}^{\infty}b_n z^n}$ is a
%locally univalent and
$K$-quasiconformal sense-preserving harmonic mapping in $\ID$ and $h\prec\varphi$, where  $\varphi$ is univalent in $\ID$. Then
\begin{equation*}
\sum_{n=1}^{\infty}|a_n|r^n+\sum_{n=2}^{\infty}|b_n|r^n\leq \dist (0,\partial\varphi(\ID))
\end{equation*}
for $|z|=r\leq r_s $,
%$|z|=r\leq r_s\approx 0.134614$,
where $r_s=r_s(k)$ is the positive real root of the equation
\begin{equation}\label{rootS}
\frac{r}{(1-r)^2}+\frac{kr^2}{(1-r^2)^2}\sqrt{\left(\frac{r^6+11 r^4+11 r^2+1}{1-r^2}\right)\left(\frac{\pi^2}{6}-1\right)}=\frac{1}{4}
\end{equation}
%\begin{equation}\label{rootS}
%(1-r)^2- 4 r \left(1+\sqrt{r (1+r)}\right)=0
%\end{equation}
in the interval $(0,1)$ and $k=(K-1)/(K+1)$. The number $r_s(k)$ cannot be replaced by the number greater than $R=R(k)$,
where $R$ is the positive root of the equation
\begin{equation}\label{eq:R1}
\frac{R(1-k+2kR)}{(1-R)^2}-k\log (1-R)=\frac{1}{4}.
\end{equation}
%The number $r_s$ is sharp and is attained by a suitable rotation of the Koebe mapping $k(z)=z/(1-z)^2$.
\end{theorem}
\begin{proof} By assumption $\varphi $ is analytic and univalent in $\ID$, $\varphi (0)=0$  and thus, by ~\eqref{GTHMS} and Theorem~\ref{SCSB}, we have
$$\dist (0,\partial\varphi(\ID))\geq \frac{|\varphi' (0)|}{4}\quad{\rm and}\quad |a_n|\leq |\varphi' (0)|n ~\mbox{ for $n\geq 1$}
$$
so that
\begin{equation}\label{AnS}
\sum_{n=1}^{\infty}|a_n|r^n\leq|\varphi' (0)|\sum_{n=1}^{\infty}n r^n=|\varphi' (0)| \frac{r}{(1-r)^2}.
\end{equation}

As in the proof of Theorem \ref{thHC}, it follows that
\begin{equation*}
\begin{split}
\sum_{n=2}^{\infty}n^2 |b_n|^2 r^{2(n-1)}&\leq k^2r^2\sum_{n=1}^{\infty}n^2 |a_n|^2 r^{2(n-1)}\\
&=k^2|\varphi'(0)|^2 r^2\sum_{n=1}^{\infty}n^4 r^{2(n-1)}\\
&=k^2|\varphi'(0)|^2 \frac{r^2\left(r^6+11 r^4+11 r^2+1\right)}{\left(1-r^2\right)^5}.
\end{split}
\end{equation*}
By using the classical Cauchy-Schwarz inequality, we deduce that
\begin{equation*}
\begin{split}
\sum_{n=2}^{\infty}|b_n|r^n &\leq \sqrt{\sum_{n=2}^{\infty}n^2 |b_n|^2 r^{2n}}\sqrt{\sum_{n=2}^{\infty}\frac{1}{n^2}}\\
& \leq k|\varphi'(0)| \frac{r^2}{(1-r^2)^2}\sqrt{\frac{r^6+11 r^4+11 r^2+1}{1-r^2}}\sqrt{\frac{\pi^2}{6}-1}.
\end{split}
\end{equation*}
Consequently, by combining  \eqref{AnS} with the last inequality, we find that
\begin{equation*}
\begin{split}
&\sum_{n=1}^{\infty}|a_n|r^n+\sum_{n=2}^{\infty}|b_n|r^n\\
&\leq \left(\frac{r}{(1-r)^2}+\frac{kr^2}{(1-r^2)^2}\sqrt{\frac{r^6+11 r^4+11 r^2+1}{1-r^2}}\sqrt{\frac{\pi^2}{6}-1}\right)  |\varphi' (0)|\\
& \leq 4\left(\frac{r}{(1-r)^2}+\frac{kr^2}{(1-r^2)^2}\sqrt{\frac{r^6+11 r^4+11 r^2+1}{1-r^2}}\sqrt{\frac{\pi^2}{6}-1}\right) \dist (0,\partial\varphi(\ID))\\
& \leq \dist (0,\partial\varphi(\ID)).
\end{split}
\end{equation*}
%As in the proof earlier theorems, it follows easily that
%\begin{equation*}
%\begin{split}
%\sum_{n=2}^{\infty}|b_n|r^n
%%&\leq\sqrt{\sum_{n=2}^{\infty}|b_n|^2 r^n}\sqrt{\sum_{n=2}^{\infty}r^n}\\
%&\leq \sqrt{\sum_{n=1}^{\infty}|a_n|^2 r^n}\sqrt{\sum_{n=2}^{\infty}r^n}\\
%&\leq  |\varphi' (0)|\sqrt{\sum_{n=1}^{\infty}n^2 r^n}\sqrt{\sum_{n=2}^{\infty}r^n}\\
%&=  |\varphi' (0)| \sqrt{\frac{r (1+r)}{(1-r)^3}}\sqrt{\frac{r^2}{1-r}}\\
%&=  |\varphi' (0)|\frac{r\sqrt{r\left(1+r\right)}}{(1-r)^2}.
%\end{split}
%\end{equation*}
%Consequently, by combining  \eqref{AnS} with the last inequality, we find that
%\begin{equation*}
%\begin{split}
%\sum_{n=1}^{\infty}|a_n|r^n+\sum_{n=2}^{\infty}|b_n|r^n
%&\leq \left [\frac{r}{(1-r)^2}+\frac{r\sqrt{r\left(1+r\right)}}{(1-r)^2}\right ]  |\varphi' (0)|\\
%& \leq 4\left [\frac{r}{(1-r)^2}+\frac{r\sqrt{r\left(1+r\right)}}{(1-r)^2}\right ] \dist (0,\partial\varphi(\ID))\\
%& \leq \dist (0,\partial\varphi(\ID))
%\end{split}
%\end{equation*}
%where the last inequality holds provided
%$$4\left [\frac{r}{(1-r)^2}+\frac{r\sqrt{r\left(1+r\right)}}{(1-r)^2}\right ] \leq 1.
%$$
This gives $r\leq r_s$, where $r_s=r_s(k)$ is the positive real root of the equation \eqref{rootS} in the interval $(0,1)$.
%Sharpness can be seen easily by considering suitable rotations of the Koebe function $k(z)=z/(1-z)^2$.

Finally, we consider the functions
\begin{equation*}
\varphi(z)=h(z)=\frac{z}{(1-z)^2}\quad\mbox{and}\quad g'(z)=k z h'(z).
\end{equation*}
So we find that
\begin{equation*}
|a_n|=n ~\mbox{ and }~ |b_n|=k\left (n+\frac{1}{n}-2\right ), \quad n\geq 2,
\end{equation*}
and thus,  we have
\begin{equation*}
\begin{split}
\sum_{n=1}^{\infty}|a_n|r^n+\sum_{n=2}^{\infty}|b_n|r^n
&=\sum_{n=1}^{\infty}n r^n+k\sum_{n=2}^{\infty}\left(n+\frac{1}{n}-2\right)r^n\\
&=\frac{r(1+k(2r-1))}{(1-r)^2}-k\log(1-r),
\end{split}
\end{equation*}
which is less than or equal to $1/4$ only in the case when  $r\leq R=R(k)$, where $R$ is the positive root of the equation \eqref{eq:R1}.
\end{proof}

\begin{corollary}\label{thHSa}
Suppose that $f(z)=h(z)+\overline{g(z)}=\sum_{n=1}^{\infty}a_n z^n+\overline{\sum_{n=2}^{\infty}b_n z^n}$ is a
%locally univalent and
sense-preserving harmonic mapping in $\ID$ and $h\prec\varphi$, where  $\varphi$ is univalent in $\ID$. Then
\begin{equation*}
\sum_{n=1}^{\infty}|a_n|r^n+\sum_{n=2}^{\infty}|b_n|r^n\leq \dist (0,\partial\varphi(\ID))
\end{equation*}
for $|z|=r\leq r_s = 0.155856\cdots$,
%$|z|=r\leq r_s\approx 0.134614$,
where $r_s$ is the positive real root of the equation
\begin{equation*}%\label{rootS}
\frac{r}{(1-r)^2}+\frac{r^2}{(1-r^2)^2}\sqrt{\left(\frac{r^6+11 r^4+11 r^2+1}{1-r^2}\right)\left(\frac{\pi^2}{6}-1\right)}=\frac{1}{4}
\end{equation*}
%\begin{equation}\label{rootS}
%(1-r)^2- 4 r \left(1+\sqrt{r (1+r)}\right)=0
%\end{equation}
in the interval $(0,1)$. The number $0.1593\cdots$ cannot be replaced by the number greater than $R=0.161353\cdots$, where $R$ is the positive root of the equation
\begin{equation*}%\label{eq:R1}
\frac{2R^2}{(1-R)^2}-\log (1-R)=\frac{1}{4}.
\end{equation*}
%The number $r_s$ is sharp and is attained by a suitable rotation of the Koebe mapping $k(z)=z/(1-z)^2$.
\end{corollary}

In view of  Corollaries \ref{thHCa} and \ref{thHSa} (see also Theorems \ref{thHC} and \ref{thHS} to propose general conjectures),
it is natural to propose in particular the following two conjectures.

\begin{conj}\label{conj1}
{\it
Suppose that $f(z)=h(z)+\overline{g(z)}=\sum_{n=0}^{\infty}a_n z^n+\overline{\sum_{n=2}^{\infty}b_n z^n}$ is a
%locally univalent and
sense-preserving harmonic mapping in $\ID$ and $h\prec \varphi$.
\begin{enumerate}
\item[{\rm (a)}]
If $\varphi$ is univalent and convex in $\ID$, then
\begin{equation}\label{conj-eq1}
\sum_{n=1}^{\infty}|a_n|r^n+\sum_{n=2}^{\infty}|b_n|r^n\leq \dist (\varphi(0),\partial\varphi(\ID))
\end{equation}
for $|z|=r\leq r_c= 0.299823\cdots$, where $r_c$ is the positive root of the equation \eqref{eq:R}.

\item[{\rm (b)}] If $\varphi$ is univalent in $\ID$, then the inequality \eqref{conj-eq1} holds
%\begin{equation*}
%\sum_{n=1}^{\infty}|a_n|r^n+\sum_{n=2}^{\infty}|b_n|r^n\leq \dist (0,\partial\varphi(\ID))
%\end{equation*}
for $|z|=r\leq r_s = 0.161353\cdots$, where $r_s$ is the positive real root of the equation \eqref{eq:R1}.
\end{enumerate}
}
\end{conj}

% \begin{center}{\sc Acknowledgements}
%\end{center}
%
%\vskip.05in

\subsection*{Acknowledgments}
The research of the first author was supported by the Foundation of Guilin University of Technology under Grant No. GUTQDJJ2018080,
the Natural Science Foundation of Guangxi under Grant No. 2018GXNSFAA050005. The research was financially supported by Hunan Provincial Key
Laboratory of Mathematical Modeling and Analysis in Engineering (Changsha University of Science $\&$ Technology).
The work of the second author is supported  in part by Mathematical Research Impact Centric Support
(MATRICS) grant, File No.: MTR/2017/000367, by the Science and Engineering Research Board (SERB),
Department of Science and Technology (DST), Government of India.

\subsection*{Conflict of Interests}
The authors declare that there is no conflict of interests regarding the publication of this paper.

%\bibliographystyle{amsplain}

%\nolinenumbers
\end{document}